\documentclass[10pt,letterpaper,notitlepage,twoside]{article}

\usepackage{amsmath,amsthm,amsfonts,amssymb,csquotes,epsfig, titlesec, epsfig, float, caption,wrapfig, mathrsfs, mathtools}
\usepackage[hidelinks]{hyperref}
\usepackage{tikz-cd}
\usepackage{tikz}

\usetikzlibrary{decorations.pathreplacing}
\usetikzlibrary{patterns}
\usepackage{verbatim}
\usepackage{accents}
\usepackage[citestyle=alphabetic,bibstyle=alphabetic,backend=bibtex, maxbibnames=10,minbibnames=5,giveninits=true]{biblatex}
\usepackage[]{enumerate}
\usepackage[american]{babel}
\usepackage{todonotes}

\usepackage{dsfont}
\usepackage{booktabs}

\usepackage{titling}
\usepackage{enumitem}

\usepackage{newpxtext,newpxmath}   

\usepackage[top=4cm, bottom=3.5cm, left=3.8cm, right=3.8cm, columnsep=20pt,includefoot]{geometry}
\usepackage{indentfirst}
\usepackage{xcolor}
\definecolor{LinkColor}{RGB}{145,80,170}
\definecolor{CiteColor}{RGB}{230,125,65}
\definecolor{FileColor}{RGB}{75,170,130}
\usepackage{hyperref}
\hypersetup{
    colorlinks=true,     
    linkcolor=LinkColor,    
    citecolor=CiteColor,     
    filecolor=FileColor,      
    urlcolor=FileColor,       
}
\usepackage[capitalize]{cleveref}
\usepackage{subcaption}

\newtheoremstyle{theorem}{}{}{\slshape}{}{\bfseries}{}{.5em}{}
\theoremstyle{theorem}
\newtheorem{theorem}{Theorem}[section]
\newtheorem{lemma}[theorem]{Lemma}
\newtheorem{proposition}[theorem]{Proposition}

\newtheorem{corollary}[theorem]{Corollary}

\newtheorem{mainthm}{Theorem}
\Crefname{mainthm}{Theorem}{Theorems}
\newtheorem{maincor}[mainthm]{Corollary}
\Crefname{maincor}{Corollary}{Corollaries}

\newtheoremstyle{remark}{}{}{}{}{\scshape}{}{.5em}{}
\theoremstyle{remark} 
\newtheorem{remark}[theorem]{Remark}
\crefname{remark}{Remark}{Remarks}
\Crefname{remark}{Remark}{Remarks}
\crefname{lemma}{Lemma}{Lemmas}
\Crefname{lemma}{Lemma}{Lemmas}

\newtheoremstyle{ex}{}{}{}{}{\scshape}{{:}}{.5em}{}
\theoremstyle{ex}
\newtheorem{example}[theorem]{Example}

\theoremstyle{definition} 
\newtheorem{definition}[theorem]{Definition} 

\titleformat*{\section}{\normalsize \bfseries \filcenter}
\titleformat*{\subsection}{\normalsize \bfseries }
\crefformat{equation}{(#2#1#3)}

\newtheoremstyle{pr}{}{}{}{}{\scshape}{{:}}{.5em}{}
\theoremstyle{pr}
\newtheorem*{pr*}{Proof}
\AtEndEnvironment{pr*}{\null\hfill\qedsymbol}

\renewenvironment{proof}[1][\proofname]{{\noindent\scshape #1: }}{\null\hfill\qedsymbol}

\usepackage{titlesec}
\renewcommand{\thesection}{\Roman{section}}
\titleformat{\section}[block]{\large\scshape}{\thesection.}{0.5em}{}
\renewcommand{\thesubsection}{\Roman{section}.\alph{subsection}}
\titleformat{\subsection}[block]{\scshape}{\thesubsection.}{0.5em}{}
\titleformat{\subsubsection}[block]{\bfseries}{}{0.5em}{}

\newcommand{\pagetitre}[8]{
\noindent\rule{\linewidth}{0.5pt}
\begin{center}
\begin{tabular}{c}
{\LARGE\textmd{\scshape #1}}\\[2pt]
{\LARGE\textmd{\scshape #2}}\\[15pt]   
{\large #3 {\scshape #4},\footnotemark}
{\large #5 {\scshape #6},\footnotemark} and 
{\large #7 {\scshape #8}\footnotemark}
\end{tabular}
\end{center}
\vspace{-0.2cm}
\noindent\rule{\linewidth}{0.5pt}}

\usepackage{fancyhdr}
\fancypagestyle{first}{
\pagestyle{fancy}

\fancyhead[C]{}
\fancyhead[RO,LE]{}
\fancyhead[LO,RE]{}
\fancyfoot[C]{}

\fancyfoot[L]{{\footnotesize Latest update: \today \\[1pt]
 $^1$\supportA\\[1pt]
 $^2$\supportB\\[1pt]
 $^3$\supportC\\[1cm]}}

}

\makeatletter
\renewenvironment{abstract}{
    \if@twocolumn
      \section*{\abstractname}
    \else\small 
      \begin{center}
        {\scshape \abstractname\vspace{-0.25cm}}
      \end{center}
      \quotation
    \fi}
    {\if@twocolumn\else\endquotation\fi}
\makeatother

\let \savenumberline \numberline
\def \numberline#1{\savenumberline{#1.}}

\makeatletter
\renewcommand{\l@section}{\@dottedtocline{1}{1.5em}{2.6em}}
\makeatother

\makeatletter
\renewcommand\tableofcontents{
	\begin{center}
	{\scshape \contentsname\vspace{-0.25cm}}
	\end{center}
	\@mkboth{\MakeUppercase\contentsname}{\MakeUppercase\contentsname}
	\@starttoc{toc}
}
\makeatother

\makeatletter
\def\namedlabel#1#2{\begingroup
   \def\@currentlabel{#2}
   \label{#1}\endgroup
}
\makeatother

\renewbibmacro{in:}{}

\newcommand{\titre}{Reverse Isoperimetric Inequalities} 
\newcommand{\titrep}{for Lagrangian intersection Floer theory} 
\newcommand{\titrepp}{Reverse isoperimetric inequalities for Lagrangian intersection Floer theory}
\newcommand{\prenomauteurA}{Jean-Philippe} 
\newcommand{\nomauteurA}{Chass\'e} 
\newcommand{\supportA}{Partially supported by the Swiss National Science Foundation (200021\_204107)}
\newcommand{\prenomauteurB}{Jeff} 
\newcommand{\nomauteurB}{Hicks} 
\newcommand{\supportB}{Supported by EPSRC grants EP/V049097/1 and EP/Z53528X/1.}
\newcommand{\prenomauteurC}{Yoon Jae Nick} 
\newcommand{\nomauteurC}{Nho} 
\newcommand{\supportC}{Supported by the Cambridge Commonwealth European and International Trust scholarship}

\newcommand{\RR}{\mathbb R}
\newcommand{\ZZ}{\mathbb Z}
\newcommand{\CC}{\mathbb C}
\newcommand{\R}{\mathbb R}
\newcommand{\C}{\mathbb C}

\renewcommand{\Im}{\text{Im}}

\DeclareMathOperator{\id}{Id}

\DeclareMathOperator{\st}{\; |\; }

\DeclareMathOperator{\Length}{Length}
\DeclareMathOperator{\Area}{Area}
\DeclarePairedDelimiter\abs{\lvert}{\rvert}

\newcommand{\absx}{\sqrt{\rho_2}}
\newcommand{\absy}{\sqrt{\rho_1}}
\newcommand{\absxsq}{\rho_2}
\newcommand{\absysq}{\rho_1}
\newcommand{\hpre}{\beta_L}

\newcommand{\injc}{r_L}

\newcommand{\injr}{{D}}

\newcommand{\rloc}{D}

\newcommand{\finj}{r_L}

\newcommand{\hght}{s}

\def\del{\partial}

\newcommand{\locf}{\beta^{A}}
\newcommand{\ff}{\beta^A}

\newcounter{ritem}
 
\newcommand{\Addresses}{{
  \bigskip
  \footnotesize

 \noindent J.-P.~Chass\'{e}, \textsc{Department of Mathematics,  ETH Z\"urich}\par\nopagebreak
  \noindent \textit{E-mail address}: \texttt{jeanphilippe.chasse@math.ethz.ch}

  \medskip 

  \noindent J.~Hicks, \textsc{School of Mathematics and Statistics,  University of St Andrews}\par\nopagebreak
  \noindent \textit{E-mail address}: \texttt{jeff.hicks@st-andrews.ac.uk}

  \medskip
  \noindent Y.J.~Nho
, \textsc{DPMMS,  University of Cambridge}\par\nopagebreak
  \noindent \textit{E-mail address}: \texttt{yjn23@dpmms.cam.ac.uk}

  \medskip

}}

\begin{document}

\newpage
\setcounter{page}{1}
\thispagestyle{first}

\pagetitre{\titre}{\titrep}{\prenomauteurA}{\nomauteurA}{\prenomauteurB}{\nomauteurB}{\prenomauteurC}{\nomauteurC}

\begin{abstract}
    \noindent We extend \citeauthor{groman2014reverse}'s reverse isoperimetric inequality to pseudoholomorphic curves with punctures at the boundary and whose boundary components lie in a collection of Lagrangian submanifolds with intersections locally modelled on $\RR^n\cap (\RR^{k}\times \sqrt{-1}\RR^{n-k})$ inside $\CC^n$. Our construction closely follows the methods used by \citeauthor{duval2016result} and \citeauthor{abouzaid2021homological} and corrects an error appearing in the latter approach.
\end{abstract}

\noindent\rule{\linewidth}{0.5pt}
\tableofcontents
\noindent\rule{\linewidth}{0.5pt}

\section{Introduction}
    \Citeauthor{groman2014reverse}'s reverse isoperimetric inequality for $J$-holomorphic curves is an important tool in the study of Floer cohomology of Lagrangian submanifolds. Let $(X, \omega, J)$ be a $2n$-dimensional symplectic manifold with a choice of compatible almost complex structure. Given a Lagrangian submanifold $L \subset X$, \cite[Theorem 1.1]{groman2014reverse} states that there exists a constant $K$ such that, for all $J$-holomorphic curves $u : (\Sigma, \partial \Sigma) \to (X, L)$ with boundary in $L$, we have a \emph{reverse isoperimetric inequality}:
\begin{equation}
 \Length(u(\partial \Sigma))\leq K \cdot \Area(u(\Sigma)),
    \label{eq:reverseIsoperimetricInequality}
\end{equation}
where length and area are given by the metric $\omega(\cdot, J\cdot)$. A different proof of this inequality was subsequently given by \citeauthor{duval2016result}~\cite{duval2016result}, whose arguments were later adapted to the setting of $J$-holomorphic polygons with boundary on a configuration of transversely intersecting Lagrangian submanifolds by \citeauthor{abouzaid2021homological}~\cite{abouzaid2021homological}.

An explicit computation of the constant appearing in  \cref{eq:reverseIsoperimetricInequality} gives a quantitative bound between the length and area of $J$-holomorphic curves in terms of the geometry of the Lagrangian $L$. However, the existence of some constant $K$ bounding the length in terms of area is sufficient for many applications. For example, consider a Liouville domain $X$ and a Lagrangian $L$ that has a cylindrical end. If $u$ is a $J$-holomorphic curve with boundary on $L$ of bounded energy, then \cref{eq:reverseIsoperimetricInequality} implies that the boundary of $u$ can only travel a fixed distance along the cylindrical end. As a consequence, there is a Gromov-compactness result for curves of this type. Such an idea has been used to ensure the compactness of moduli spaces appearing in the definition of certain quilted Floer cohomology groups \cite{torricelli2022projective}. Another application comes from family Floer theory \cite{abouzaid2021homological}, where the convergence of the Floer differential for a non-unitary local system can be proven by showing that the norm of the monodromy of the local system along the boundary of a curve is bounded from above by the perimeter. Similarly, the reverse isoperimetric inequality is useful in adiabatic degeneration situations for multi-graph Lagrangian submanifolds with caustics, where one needs to separate the domain of holomorphic disks into regions that degenerate to Morse flow-trees and regions near the caustics. 

In some cases, we can derive tight bounds for the constant $K$ in \cref{eq:reverseIsoperimetricInequality}, which endows Floer cohomology with additional structure. For instance, in \cite{hicks2019tropical}, the second author noticed a relationship between the areas of specific $J$-holomorphic strips with boundaries on tropical Lagrangian submanifolds and the affine lengths in tropicalization. This observation can be restated in terms of a bound for the constant $K$ in terms of tropical geometry.

When the boundary Lagrangian $L$ is an embedded Lagrangian submanifold, the constant $K$ roughly measures the radius of a standard symplectic neighbourhood of $L$. In this note, we replace $L$ with a collection $\{L_i\}_{i=1}^m$ of Lagrangian submanifolds with pairwise disjoint locally standard intersections (\cref{def:locallyStandard}).

A reverse isoperimetric inequality for $J$-holomorphic polygons with boundary on transversely intersecting Lagrangian submanifolds had previously appeared in \cite[Appendix A.1]{abouzaid2021homological}. However, the construction of a weakly plurisubharmonic function in that paper contains an error which we describe in \cref{rem:abouzaid}. Therefore, our result also corrects the result appearing there.

\subsection*{Results and strategy of proof}
The results that we prove and the method of proof follow closely that of \citeauthor{duval2016result}~\cite{duval2016result}. Let $(X, \omega, J, g)$ be a 2n-dimensional almost K\"ahler manifold.
Let $S$ be a Riemann surface with marked boundary points whose boundary arcs $\{C_i\}_{i=1}^m$ are labelled by the collection of embedded Lagrangian submanifolds $\{L_i\}_{i=1}^m$. \par

We will need to restrict ourselves to Lagrangian submanifolds intersecting nicely. More precisely, we require that every point in $L_i\cap L_j$, $i\neq j$, is in a chart $\phi:\C^n\to M$ preserving both the symplectic and almost complex structure and sending $\sqrt{-1}\R^n$ to $L_i$ and $\R^{n-k}\times\sqrt{-1}\R^k$ to $L_j$. Here, we identify $\C^n$ with $\R^n\times\sqrt{-1}\R^n$. We call such intersections locally standard,~---~we refer the reader to Subsection~\ref{subsec:localModels} below for details and examples where this happens. 

\begin{mainthm} \label{thm:reverseIsoperimetricInequality}
 Suppose that the intersections $L_i\cap L_j$ are pairwise disjoint and locally standard, and let $B$ be any open neighbourhood of $\cup_{ij} L_i\cap L_j$. There exists constants $K, \finj>0$, depending only on $(X,\omega,J,L_i)$ and $B$ so that, for any $J$-holomorphic curve $u:S\to X$ sending the boundary arc $C_i$ of $\del S$ to $L_i$, $1\leq i\leq m$, and $0<\hght<\finj$:
    \[ \hght\cdot \Length_g(\Im(\partial u)\cap B^c)\leq K\cdot \Area_g(\Im(u)\cap U_{\hght}) .\]

\end{mainthm}
Here, $U_s=\bigcup_i N_s(L_i)$, where $N_{\hght}(L_i)$ is a tubular neighbourhood of $L_i$ of radius $\hght$, and $B^c=X-B$. Furthermore, $\finj$ is a constant smaller than the minimal radial injectivity radius of the Lagrangians, suitably modified to take the intersection locus into account.

By modifying the almost complex structure to make transverse intersections locally standard (\cref{prop:transverseLocallyStandard}), we get the following result.
\begin{maincor} \label{cor:maincor}
 For any collection of Lagrangian submanifolds $L_1, \ldots, L_m\subset X$ which have pairwise disjoint transverse intersections, there exists a choice of almost complex structure so that a reverse isoperimetric inequality {\normalfont \`a la} (\ref{eq:reverseIsoperimetricInequality}) holds. More precisely, (\ref{eq:reverseIsoperimetricInequality}) holds with the caveat that the length is only measured in the complement of some neighbourhood of $\cup_{ij}( L_i\cap L_j)$.
    
\end{maincor}

Note that only being able to estimate the length outside some fixed neighbourhood of the intersection locus $\cup_{ij} L_i\cap L_j$ is enough for most applications such as family Floer theory. For example, it still implies that, given a uniform energy bound, $J$-holomorphic polygons with boundary along $L_1,\dots L_m$ with energy below that bound must stay a bounded distance away from the intersection locus. 

\begin{remark} \label{rem:c0-dense}
 As one will see below, given Lagrangian submanifolds $L_1,\dots, L_m$ with pairwise disjoint transverse intersections and an $\omega$-compatible almost complex structure $J'$, the almost complex structure $J$ satisfying the conclusions of \cref{cor:maincor} can be taken to be $C^0$-close to $J'$ and equal to $J'$ outside $B$.
\end{remark}

\begin{remark} \label{rem:immersed}
 Our proof can also be modified to incorporate teardrops and holomorphic disks in the Lagrangian projections of Legendrians that allow us to compute Legendrian contact homology. Indeed, we can easily allow Lagrangians with transverse self-double intersections, as long as we assume that the self-intersection locus is clean, does not intersect the intersection locus with other Lagrangians, and can be made locally standard. In that case, we can allow $L_i=L_{i+1}=L$ with the corresponding marked point of $\del S$ being sent to a self-intersection point.
\end{remark}

The proof of \cref{thm:reverseIsoperimetricInequality} follows the lines of \cite{duval2016result}, who observes that the square of the distance function  $\rho: N_{\injr}(L)\to \RR$ can be perturbed to give a strictly plurisubharmonic function $h:  N_{\injr}(L)\to \RR$ which vanishes on $L$ with weakly plurisubharmonic square root. In that vein, we produce a function $h: U_{\injr}\to \RR$ which is a small perturbation of $\rho_i: U_{\hght}\to \RR$ away from a neighbourhood of the intersection locus $\cup_{i,j} L_i\cap L_j$ and has weakly plurisubharmonic square root in a neighbourhood of the intersection locus. The proof of \cref{thm:reverseIsoperimetricInequality} can be broken into three steps:
\begin{description}
    \item[\cref{subsec:localModels}] Constructing local models for $h$ near the intersection locus. When $L_i$ and $L_j$ have intersections of the form given by \cref{def:locallyStandard}, we show that $\sqrt{\rho_i\rho_j}$ is weakly plurisubharmonic in a neighbourhood of the intersection locus. \label{item:step1}
    \item[\cref{subsec:interpolation}] Showing that we can interpolate between the local models near the intersection and the function $\rho_i$ away from the intersection while remaining weakly plurisubharmonic. \label{item:step2}
    \item[\cref{subsec:RII}] Modifying \citeauthor{duval2016result}'s proof to instead use the function $h: U_s\to \RR$. \label{item:step3}
\end{description}
We delay the proofs in \cref{subsec:localModels}  that the local models of $h$ are plurisubharmonic until \cref{sec:pshProofs} to improve readability. 

\subsection*{Acknowledgements}
Some initial computations for this paper were performed with the aid of Mathematica \cite{Mathematica}. The authors would like to thank M.~Abouzaid for encouraging them to write this note and an anonymous referee for their detailed and thoughtful comments. The third author would like to thank his supervisor A.~Keating for her useful advice and feedback. 
 \section{The reverse isoperimetric inequality}
        \label{sec:interpolation}
        \subsection{The local model near the intersection}
\label{subsec:localModels}
We restrict ourselves to Lagrangian submanifolds whose intersections have particularly nice local models. 
\begin{definition}
 We say that the intersection between Lagrangian submanifolds $L, L'$ is \emph{locally standard} if at every point $x\in L\cap L'$, there exist a chart $U\subset (\CC^n, \omega_{\CC^n}, J_{\CC^n})$, $\phi: U\to X$; and choice of $0 \leq k \leq n$ so that 
    \begin{align*}
        \phi(0)=x && \phi^{-1}(L)=\sqrt{-1}\RR^n&& \phi^{-1}(L')=\RR^{n-k}\times \sqrt{-1} \RR^k\\
        \phi^*J = J_{\CC^n} && \phi^*\omega= \omega_{\CC^n}.
    \end{align*}
    \label{def:locallyStandard}
\end{definition}

Since intersection points of transversely intersecting Lagrangian submanifolds admit standard neighbourhoods, we directly get the following result.
\begin{proposition}
 For any pair of transversely intersecting Lagrangian submanifolds $L$ and $L'$ in $(X, \omega)$, there exists a choice of compatible almost complex structure so that the intersection is locally standard.
    \label{prop:transverseLocallyStandard}
\end{proposition}
Observe that there exist locally standard clean intersections.
\begin{example}
 The following construction comes from \cite[Remark, page 9]{cleanintersectioncompactness}. Suppose $K=L_0\cap L_1$ admits a flat metric. Realize the neighbourhood of $K$ as $T^{\ast}L_0$ and $L_1$ as the conormal $NK$.
 Choose a metric $g$ on $L$ such that it is flat in the neighbourhood of $K$ in $L_0$, makes $K$ totally geodesic, and restricts to a globally flat metric on $K$. Let $J$ be the almost complex structure on $T^{\ast}L_0$ induced by the connection on $T^{\ast}L_0$ given by $g$. Taking geodesic normal coordinate sending $K$ to $\mathbb{R}^k\subset \mathbb{R}^n$, we get open charts satisfying the conditions in Definition \ref{def:locallyStandard}. 
\end{example}

By Bieberbach's theorem, any compact flat Riemannian manifold is a finite quotient of the torus. While this puts a restriction on the topology of the intersection, intersections of this form naturally appear in computations motivated by mirror symmetry.
\begin{example}
 Following the notation from \cite{hicks2019tropical}: let $V_1, V_2\subset Q$ be two tropical subvarieties in an affine manifold $Q$.
     
 Suppose that they intersect cleanly in a collection of points $V_1\cap V_2=\{q_1,\ldots, q_k\}$. Whenever $V_1, V_2$ admit tropical Lagrangian lifts $L_{V_1}, L_{V_2}\subset T^*Q/T^*_\ZZ Q$, then the intersection $L_{V_1}\cap L_{V_2}$ is locally standard and is the union of $k$ disjoint tori of dimension  $\dim(Q)-\dim(V_1)-\dim(V_2)$.
\end{example}

For ease of exposition, we will now assume that we are studying $J$-holomorphic curves with boundary on two Lagrangians $L_1$, $L_2$ with locally standard intersections.
The local model for this situation is the intersection in $\CC^n$ of the Lagrangian planes $L^{loc}_1=\{x_i=0\;|\;1\leq i\leq n\}$ and $L^{loc}_2=\{x_i=0,y_j=0\;|\;1\leq i\leq k,k+1\leq j\leq n\}$ for some $0\leq k\leq n$~---~the case $k=0$ corresponds to a transverse intersection. \par

In what follows, we fix $n$ and $k$ as above and consider the functions
\begin{align}
    \rho_1(x,y):=\sum_{i=1}^n x_i^2 \qquad\text{and}\qquad \rho_2(x,y):=\sum_{i=1}^k x_i^2+\sum_{i=k+1}^n y_i^2
    \label{eq:defOfRho}
\end{align}
on $\CC^n=\RR^n_x\oplus \sqrt{-1}\RR^n_y$. Note that $L^{loc}_1=\{\rho_1=0\}$ and $L^{loc}_2=\{\rho_2=0\}$. 

\begin{proposition} \label{prop:prod_psh-clean}
 The functions $\sqrt{\rho_1\rho_2}$ and $\rho_1\rho_2$ are weakly plurisubharmonic on the standard chart $U_x$ at $x\in L^{loc}_1\cap L^{loc}_2$. Furthermore, outside of some variety $V$ such that $V\cap (L^{loc}_1\cup L^{loc}_2)=L^{loc}_1\cap L^{loc}_2$, $\rho_1\rho_2$ is strictly plurisubharmonic.
\end{proposition}
We delay the proof until \cref{sec:pshProofs}. We note however that the set $V$ is the precise reason why we need to suppose the existence of standard charts about intersections. Indeed, without it, we do not have an obvious choice of plurisubharmonic function, since $\absysq\absxsq$ might no longer be~---~even weakly~---~plurisubharmonic near $V$. 

\begin{remark}
    \label{rem:abouzaid}
 A different approach to constructing the local model was proposed in \cite[Appendix A.1]{abouzaid2021homological}; unfortunately, this approach contains a gap. The method uses a cutoff function, which is employed to excise a small neighbourhood of the intersections before applying the argument from \cite{duval2016result}. The proposed local model for plurisubharmonic function is  $\rho=\chi(x_1)\cdot \abs{y}^2$, where $\chi$ is a cutoff function that is convex and non-negative. Unfortunately, this will usually not be weakly plurisubharmonic.  If we restrict to $n=2$, the determinant of the Levi matrix of $\rho$ is
    \[4\chi^2-(2y_2\chi')^2+2|y|^2\chi\chi''\]
 Restricting to where $y_1=0, y_2=1$ we obtain the necessary inequality $2\chi^2+\chi\chi''\geq 2(\chi')^2$. Since $\chi'$ dominates $\chi$ as $x_1\to 0$, we can simplify to the condition that 
    \[\chi\chi''\geq 2(\chi')^2,\]
 which is not satisfied, for example, by the standard choice of cutoff function $\exp(-x^{-1})$. Experimentally, we couldn't find a choice of cutoff function that satisfies this relation.
\end{remark}

For $x\in L_1\cap L_2$, let $U_x$ denote the standard neighbourhood provided by \cref{def:locallyStandard} that identifies $L_i$ with $L^{loc}_i$.
In order to use the argument from \cite{duval2016result} we must find a small neighborhood of the $L_i$ for which various functions are plurisubharmonic.  The requirements that we place on this neighborhood size (which we denote by $r_L$) are
\begin{enumerate}[series=rconstraints,label=\emph{(r.\roman*}), ref=\emph{(r.\roman*})]
    \item $r_L$ is less than the injectivity radius of the Lagrangians $L_i$. \label{rlInject}
    \item  $\tilde \rho_i$ is (strictly) plurisubharmonic on $N_{r_L}(L_i)$~---~for a proof of the plurisubharmonicity of $\tilde{\rho}_i$ near $L_i$, see for example Proposition~2.15 of~\cite{CieliebakEliashberg2012}, where plurisubharmonicity is referred to as $J$-convexity;
    \item it is smaller than the Lipschitz constant of each $L_i$, i.e.\ the largest $c>0$ such that for all $x\in L_i$, $B_{c}(x)\cap L_i$ is contractible and $d_L(p,q)\leq c^{-1}d_X(p,q)$ for all $p,q\in B_{c}(x)\cap L_i$;
    \item it is smaller than half the minimal distance between connected components of the intersection locus $L_1\cap L_2$;
    \item the tube $\{\max\{\tilde \rho_1,\tilde \rho_2\}\leq r_L^2\}$ is contained in a finite covering of $L_1\cap L_2$ by locally standard charts $\{\phi_j:U_{x_j}\to M\}$ such that $\phi_j(U_{x_j}\cap \{\max\{\tilde \rho_1,\tilde \rho_2\}\leq r_L^2\})$ is convex for all $j$. \label{rlTube} 
\end{enumerate}

Let $U_{r}:=\bigcup_i N_{r}(L_i)$  and $B_{r}=\bigcup_{x\in L_1\cap L_2}\{x'\in U_x \st \max\{\tilde{\rho}_1,\tilde{\rho}_2\}<{r}^2\}$ for $r\leq \injc$. Our aim in this section is to show the following modification of Duval's statement. To ensure that the functions we consider are plurisubharmonic over the local charts at the intersections, we will have to take a neighborhood of slightly smaller size $Dr_L$ where $D\in (0, 1)$.

\begin{proposition}\label{prop:Duvalconstruction}
For $\finj$ satisfying \ref{rlInject}--\ref{rlTube}  and $\rloc\in (0,1)$ satisfying \ref{rlHessian}--\ref{rlMetricEquivalence} below, there exist constants $C_1$, $C_2$ $C_3$ with the following property. There is a nonnegative function $h: U_{\rloc\finj}\to \RR$ such that the following holds: 
\begin{enumerate}[label=(\arabic*)]
    \item $h$ vanishes precisely on $L_1\cup L_2$\label{cond:1};
    \item $\sqrt{h}$~---~and thus $h$~---~is weakly plurisubharmonic on $U_{\rloc\finj}$ and $h$ is strictly plurisubharmonic on $U_{\rloc\finj}\setminus B_{\finj}$;\label{cond:2} 
    \item the pseudometric $k=dd^ch(\cdot,\sqrt{-1}\cdot)$ is dominated by $C_1g$; \label{cond:3}
    \item on $U_{\rloc\finj}\setminus B_{\finj}$, the pseudometric $k_{\finj}$ is metric-equivalent to $g$ with $C_2^{-1}g\leq k\leq C_2g$; \label{cond:4}
    \item $C_3\sqrt{h}\geq\abs{\nabla h}$ outside $B_{\finj}$.  \label{cond:5}
\end{enumerate}
\end{proposition}
The properties that are hard to establish are \cref{cond:2}-\cref{cond:4}. In \cite{duval2016result}, the function $h$ is constructed by taking $\sqrt{h}=C\rho+\sqrt{\rho}$ for some large constant $C>0$ that depends only on $(X,J,L)$. The key idea of this paper is to use the local model for $L_i^{loc}$s to find a locally defined function satisfying \cref{prop:Duvalconstruction} near the clean intersections and interpolate this function to Duval's function, for some potentially larger $C>0$. For the statement of \cref{thm:reverseIsoperimetricInequality}, we simply re-define $\finj$ to be $\finj \rloc$. 

\subsection{Local Model}\label{subsection:localmodel}

The aim of this section is to show the analogue of Proposition \ref{prop:Duvalconstruction} in the special case where we are in the local neighbourhood. Our function will be constructed so that near the boundary of the local neighbourhood, our function agrees with that of Duval. 

In order to do this, we need to make some preliminary choices. Fix some $A>0$ 
, and define $V_A^{loc}:=\{\rho_1< {A}^2\}\cup \{\rho_2< {A}^2\}$; this $A$ will dictate the size of local neighborhood we want to place intersection points of the $L_i$.\footnote{ As the bound that we produce in this subsection does not depend on any of the constraints we have placed on $r_L$, we introduce the new constant $A$. }
Fix $\chi^A$ be a smooth nondecreasing function such that $\chi^A=t$ for $0\leq \sqrt{t} \leq \frac{A}{2}$ and $\chi^A(t)=A^2$ for $\sqrt{t}\geq \frac{3A}{4}$. Set $\beta^A:=\chi^A(\rho_1)\chi^A(\rho_2)$. As in \cref{fig:localmodel}, our function $\locf$ interpolates between the function $\rho_1\rho_2$ and $A^2\rho_2$ within the region $\frac{A}{2}<\sqrt{\rho_1}<\frac{3A}{4}$.
\begin{proposition}
    \label{lem:sqrtpsh} There exist constants $C_0,C_1,C_2$ and some small $D\in (0,\frac{1}{2}]$ such that 
\begin{enumerate}
    \item The function $h_{loc}=\big(\sqrt{\locf}+C_0\locf\big)^2$ has weakly plurisubharmonic square root on $V^{loc}_{DA}$.  
    \item The pseudometric induced by $h_{loc}$ is $C_1$-dominated from above by the Euclidean metric on $V^{loc}_{DA}$, and it is $C_2$-equivalent to it on $\{\sqrt{\rho_1}> \frac{A}{2}\}\cap V_{{DA}}^{loc}$ and $\{\sqrt{\rho_2}> \frac{A}{2}\}\cap V_{{DA}}^{loc}$.
    \end{enumerate} 
\end{proposition}    
    In order to show \cref{lem:sqrtpsh}, we first need the following intermediary statement. 
\begin{proposition}\label{prop:pshconstruction}
  For every $A>0$, there is a $D\in (0,\frac{1}{2}]$~---~satisfying \ref{rlHessian} below~---~such that the restriction of $\beta^A$ to $V_{{DA}}^{loc}$ is weakly plurisubharmonic.
  Furthermore, $\beta^A$ vanishes at least up to first order on $L_1^{loc}=\{\rho_1=0\}$ and $L_2^{loc}=\{\rho_2=0\}$. The pseudo-metric obtained from $\beta^A$ is dominated from above by the Euclidean metric everywhere, and it is equivalent to it on $\{\sqrt{\rho_1}> \frac{A}{2}\}\cap V_{{DA}}^{loc}$ and $\{\sqrt{\rho_2}> \frac{A}{2}\}\cap V_{{DA}}^{loc}$. 
\end{proposition}
\begin{figure}
    \centering
    \begin{tikzpicture}[scale=.75, rotate=90]
\fill[red!20]  (-2.5,-2) rectangle (3.5,-6);
\fill[red!20]  (2.5,-7) rectangle (-1.5,-1);
\draw[thick] (0.5,-1) -- (0.5,-7);
\draw[thick] (-2.5,-4) -- (3.5,-4);
\node at (2.5,-4) {$L_1$};
\node at (0.5,-2) {$L_2$};
\draw   (0.5,-0.5);
\draw[<->] (0.5,2) -- (0.5,-0.5) -- (2.5,-0.5);
\draw [decorate,
    decoration = {brace}] (-2.5,-4) -- node[below]{$\sqrt{\rho_1}=A$}  (-2.5,-2);
\fill[blue, opacity=.5] (-0.5,-3.5) -- (-0.5,-4.5) -- (0,-4.5) -- (0,-5) -- (1,-5) -- (1,-4.5) -- (1.5,-4.5) -- (1.5,-3.5) -- (1,-3.5) -- (1,-3) -- (0,-3) -- (0,-3.5) -- cycle;
\node[left] at (0.5,2) {$\chi\circ\rho_2$};
\node at (2.5,-0.5) {};
\node[below] at (0.5,1.5) {$1$};
\node at (3.85,-3.75) {$V_{D}$};
\draw (0.5,-0.5) .. controls (0.95,-0.5) and (1.3,-0.2) .. (1.5,0) .. controls (1.75,0.25) and (1.85,1.5) .. (2,1.5);
\draw (2,1.5) -- (2.5,1.5);
\draw[dotted] (-0.5,-1) -- (-0.5,-7) (1.5,-7.5) -- (1.5,1.5) (-2.5,-3) -- (3.5,-3) (3.5,-5) -- (-2.5,-5);
\node[right ] at (1.5,-7.5) {$\sqrt{\rho_2}=A/2$};
\node[above] at (2.5,-0.5) {$\sqrt{\rho_2}$};
\fill[blue, opacity=.5]  (-0.5,-3.5) rectangle (-2.5,-4.5);
\fill[blue, opacity=.5]  (1.5,-3.5) rectangle (3.5,-4.5);
\fill[blue, opacity=.5]  (1,-5) rectangle (0,-7);
\fill[blue, opacity=.5]  (0,-3) rectangle (1,-1);
\end{tikzpicture}     \caption{Local model for the intersection between two Lagrangian submanifolds. The red region represents the region where we have our local model.
 The blue regions is $V_{{DA}}^{loc}$, which is divided into three cases
 by the dashed lines labelling when $\sqrt{\rho_i}=A/2$.}
    \label{fig:localmodel}
\end{figure}
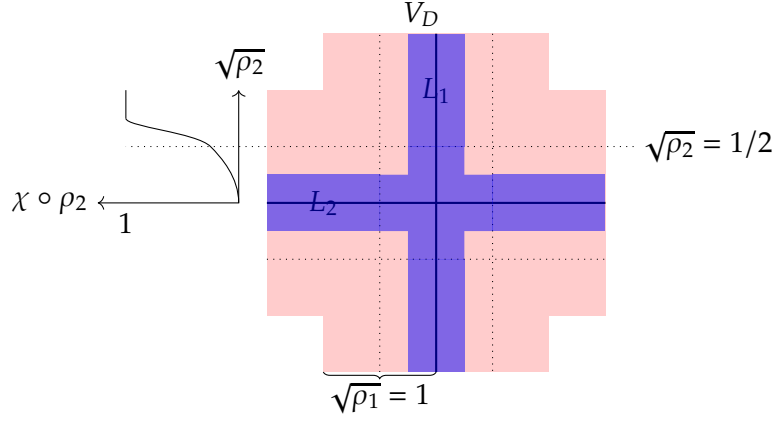
\begin{proof}
 
We first show that $\beta^A$ satisfies \cref{prop:pshconstruction} on a region $V_{{DA}}^{loc}$ by decomposing into three subregions:
\begin{itemize}
    \item Whenever $\absx< A/2$ and $\absy< A/2$, $\beta^A=\absxsq\absysq$, which is plurisubharmonic by \cref{prop:prod_psh-clean}. 
    \item Suppose $\absy\geq \frac{A}{2}$ and $\absx<\frac{A}{2}$. Then $\beta^A$ has the form $\absxsq\chi^A(\absysq)$. So we get
    \begin{align}\label{eq:Hessianform}
 dd^c \beta^A(\cdot,\sqrt{-1}\cdot)=& \chi^A(\absysq)dd^c(\absxsq)+2\absx\Big(d(\chi^A(\absysq))\wedge d^c\absx+ d\absx\wedge d^c(\chi^A(\absysq))\Big)\nonumber\\&+\absxsq\cdot dd^c(\chi^A(\absysq))\nonumber\\
 =&2\chi^A(\absysq)\id+O(\absx).
    \end{align}
 To show that the function $\beta^A$ is plurisubharmonic after shrinking $\absx$, we need to show that the form \eqref{eq:Hessianform} is non-negative.
 Observe that  $\chi^A(\absysq)\geq \frac{A^2}{4}$ for $\absy\geq\frac{A}{2}$, so we may impose condition:  
 \begin{enumerate}[series=lconstraints,label=\emph{(l.\roman*}), ref=\emph{(l.\roman*})]
    \item Require that  $D$ sufficiently small so that \eqref{eq:Hessianform} is positive definitive when $\absx<DA$.\label{rlHessian}
 \end{enumerate} 
    \item The argument is exactly the same for $\absx\geq \frac{A}{2}$ and $\absy<\frac{A}{2}$.
\end{itemize}

The comparison with the Euclidean metric $g_0$ follows from \eqref{eq:Hessianform}.
\end{proof}

We will now strengthen \cref{prop:pshconstruction} and show \cref{lem:sqrtpsh} by replacing $\locf$ with $h_{loc}=(\sqrt{\locf}+C\locf)^2$ for some large constant $C>0$, so that it has weakly plurisubharmonic square root everywhere. This is in analogy with the modification in \citeauthor{duval2016result}, where he replaces the squared distance function $\rho$ with $\big(\sqrt{\rho}+C\rho\big)^2$. The two functions can then be arranged to agree on the overlaps.

\begin{proof}[Proof of \cref{lem:sqrtpsh}].
Observe that by \cref{prop:prod_psh-clean}, we have (weak) plurisubharmonicity of $\sqrt{\ff}$ over $\{\absy<D A\}\cap \{\absx<D A\}$. Furthermore, outside $\{\absy<A\}\cap \{\absx<A\}$, $\ff$ agrees with $A^2\rho_i$ and so it suffices to handle the remaining region, i.e.\ $\{DA<\absy<A, \absx<DA\}\cup\{\absy<DA, DA<\absx<A\}$.

To handle this case, let $\absy<DA\leq \frac{A}{2}$ and $DA<\absx<A$. The function $\ff$ has the form
\[\sqrt{\ff}=\absy \sqrt{\chi^A(\absxsq)},\]
so that
\begin{align*}
 dd^c\sqrt{\ff}= \sqrt{\chi^A}\cdot dd^c \absy+d\absy\wedge d^c \sqrt{\chi^A}+d\sqrt{\chi^A}\wedge d^c\absy+\absy dd^c \sqrt{\chi^A}.
\end{align*}

Note that the only term that might become unbounded as $\absy\to 0$ is $\sqrt{\chi^A}dd^c\absy$ since its expression can contain negative powers of $\absysq$. However, we know that the form $dd^c\absy(\cdot,\sqrt{-1}\cdot)\leq 2A^{-1}\sqrt{\chi^A}dd^c\absy(\cdot,\sqrt{-1}\cdot)$ is positive semidefinite. So the only term that contains negative powers of $\absy$ must already be positive semi-definite. Furthermore, the last three terms may be negative, but they do not contain negative powers of $\absy$. Therefore, their negative contribution may be canceled out by adding some multiple of $dd^c\ff$. 

In other words, for $C_0$ large enough, we choose
 \begin{enumerate}[resume=lconstraints,label=\emph{(l.\roman*}), ref=\emph{(l.\roman*})]
    \item $D>0$ small enough so that both terms on the right hand side of
\begin{align*}
 dd^c\sqrt{\ff}+C_0dd^c\ff=& \sqrt{\chi^A}\cdot dd^c \absy\\&+ \Big(C_0dd^c\ff+ d\absy\wedge d^c \sqrt{\chi^A}+d\sqrt{\chi^A}\wedge d^c\absy+\absy dd^c \sqrt{\chi^A}\Big).
\end{align*}
are positive semidefinite on $V_{DA}^{loc}$. 
\label{rlPositiveSemiDefinite}
\end{enumerate} 
This shows that the function satisfies 1. It is also straightforward to see that this implies the first point of 2, i.e. the pseudometric is dominated by the Euclidean metric $g_0$. 

The equivalence with $g_0$ on $V_{{DA}}^{loc}\setminus \{\max\{\absy,\absx\}< \frac{A}{2}\}$ for some $C_2>0$ is subtler. Observe that, in our region, the function $(\sqrt{\ff}+C_0\ff)^2$ is equal to 
\[\absysq\chi^A(\absxsq)+C_0^{2}\absysq^2\chi^A(\absxsq)^2+2C_0\absy^3\sqrt{\chi^A(\absxsq)}^3.\]
The Hessian of $(\sqrt{\ff}+C_0\ff)^2$ is thus of the form $\chi^A(\rho_2)\id+O(\absy)$. Now, since $\chi^A(t)>\frac{A^2}{4}$ for $\sqrt{t}\geq \frac{A}{2}$, it suffices to ask that
\begin{enumerate}[resume=lconstraints,label=\emph{(l.\roman*}), ref=\emph{(l.\roman*})]
   \item take $D>0$ sufficiently small so that the term $O(\absy)$ is smaller than $\frac{A^2}{4}\id$\label{rlMetricEquivalence}
\end{enumerate}
to ensure that the Hessian is equivalent to $g_0$. Treating the other case similarly finishes the proof.
\end{proof}

\subsection{Interpolating from local model near intersections to the distance function}
\label{subsec:interpolation}
Using \cref{prop:pshconstruction}~---~and the notation from Proposition~\ref{prop:Duvalconstruction}~---~we now interpolate the distance function and the local functions. We first define the global analogue of $\locf$.

Take $x\in L_1 \cap L_2$ and standard chart $\phi_x:U_x\to\C^n$. By taking a smaller chart if necessary, we may suppose that $\phi_x(U_x)=\{\absy <A_x\}\cap \{\absx<A_x\}\cap (-A_x,A_x)^{2n}$ for some $A_x>0$. By definition of $\injc$, the intersection $L_1\cap L_2$ can be covered by finitely such charts $\{\phi_{x_j}:U_{x_j}\to\C^n\}$ with $A_{x_j}<\injc$. We then set $A=\min_j A_{x_j}$, and let $\beta_{j}:=\locf\circ \phi_{x_j}$. Then $\beta_x=\tilde{\rho}_i$ near the part of the boundary where $\sqrt{\tilde{\rho}_j}=A$, $i\neq j$. Furthermore, since $\beta_x$ only depends on the distance from $L_1\cap L_2$, the neighbourhood of which is flat, we see that $\beta_x$ does not depend on the precise choice of the chart. 

Define the function $\hpre:N_{\injc}(L)\to \RR_{\geq 0}$ by the formula
    \begin{align*}
 \hpre:=\begin{cases}
 \beta_j(p) & \text{ if $p\in U_{x_j}$}\\
        A^2    \tilde{\rho}_i(p) & \text{ if $p\in N_{\injc}(L_i)\setminus \bigcup_{j}U_{x_j}$}.
        \end{cases}
    \end{align*}
We then observe that $\sqrt{\hpre}={\absy\absx}$ on a neighbourhood of the clean intersection. Furthermore, on $N_{\injc}(L_i)\setminus B_{\injc}$, we have $\hpre=A^2\tilde{\rho}_i$.

\begin{proof}[Proof of \cref{prop:Duvalconstruction}]
Just as in \cref{prop:pshconstruction}, we deduce that our function $\hpre$ satisfies condition \ref{cond:1}. Furthermore, the pseudo-metric obtained from $\hpre$ is dominated from above by $g$ everywhere, and it is equivalent to it on $\{\tilde{\rho}_1> \frac{A^2}{4}\}\cap N_{\rloc\injc}(L_2)$ and $\{\tilde{\rho}_2> \frac{A^2}{4}\}\cap N_{\rloc\injc}(L_1)$ for small enough  $0<\rloc<\frac{A}{2}$.

Following Duval's approach~\cite{duval2016result}, we prove that for possibly small $\rloc$, there is some constant $C_0$ 
such that $h_{\injc}=\left(\sqrt{\hpre}+C_0\hpre\right)^2$ restricted to $N_{D\injc}(L)$ respects conditions \ref{cond:1}--\ref{cond:5}. Indeed, let $C_0,C_1$ be larger than the constants obtained in \cref{lem:sqrtpsh}, for each standard neighbourhood of the intersection locus. In that case, near the intersection locus, we simply have $h_{\injc}=(\sqrt{\ff}+C_0\ff)\circ\phi_x$ in standard charts, for which the statement reduces to that of \cref{lem:sqrtpsh}. Therefore, it suffices to show these conditions hold on $N_{\rloc\injc}(L_i)\setminus B_{\injc}$, where $\hpre=A^2 \tilde{\rho}_i$. This is the case of a single Lagrangian submanifold, which is handled by Duval~\cite{duval2016result}. This finishes the proof of \cref{prop:Duvalconstruction}.

\end{proof}

 \subsection{Proof of \cref{thm:reverseIsoperimetricInequality}}
    \label{subsec:RII}
    Let $\finj$ and $D$ be as in \cref{prop:Duvalconstruction}. Let $U_s= \cup_i N_{s}(L_i)$ and $U^h_s:=\{h\leq s^2\}$ for $s<D\injc$.
\begin{corollary} \label{cor:reverseisoineq}
Let $B$ be a neighbourhood of $L_1\cap L_2$.  There exists some $K_0>0$ such that
\begin{align}\label{eq:our_reverseisoineq}
\frac{K_0}{s}\Area_g(C\cap U_s)\geq \Length_g(\partial C\cap B^c)
\end{align}
\end{corollary}
\begin{figure}
    \centering
    \begin{tikzpicture}
\node[right] at (4,-4) {$L_1$};
\node[above] at (0.5,-0.5) {$L_2$};
\node[left] at (-1,-1.5) {$L_3$};

\draw[line width=1cm, blue!20] (0.5,-0.5) -- (0.5,-5.5);
\draw[line width=1cm, blue!20] (-1,-4) -- (4,-4);
\draw[line width=1cm, blue!20] (-1,-1.5) .. controls (-0.5,-1.5) and (1.5,-1.5) .. (2,-1.5) .. controls (2.5,-1.5) and (3,-2) .. (3,-2.5) .. controls (3,-3) and (3,-3.5) .. (3,-5.5);

\fill[pattern color=red, pattern=north west lines]  (0,-1) rectangle (1,-2);
\fill[pattern color=red, pattern=north west lines](0,-3.5) rectangle (1,-4.5);
\fill[pattern color=red, pattern=north west lines]  (2.5,-3.5) rectangle (3.5,-4.5);

\draw[thick] (0.5,-0.5) -- (0.5,-5.5);
\draw[thick] (-1,-4) -- (4,-4);
\draw[thick] (-1,-1.5) .. controls (-0.5,-1.5) and (1.5,-1.5) .. (2,-1.5) .. controls (2.5,-1.5) and (3,-2) .. (3,-2.5) .. controls (3,-3) and (3,-3.5) .. (3,-5.5);

\draw [decorate,    decoration = {brace}] (-1,-4) -- node[left]{$s$}  (-1,-3.5);

\node at (-0.5,-4.5) {$U_s$};
\node at (0.75,-3.75) {$B$};
\end{tikzpicture}     \caption{The neighbourhood $U_s$ is highlighted in blue, while the region $B$ (which is excluded in computing the length) has red hash lines.}
\end{figure}
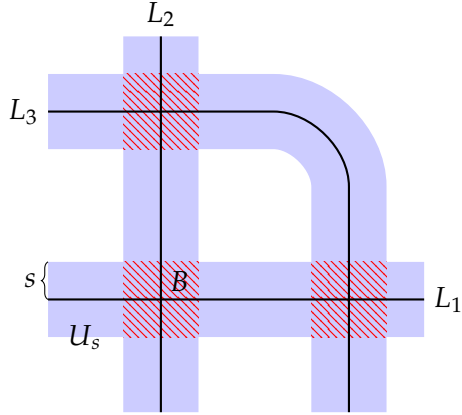
\begin{proof}
The proof is same as in \cite{duval2016result}. We pass to the $h$-metric using metric equivalence and domination, and then prove the inequality for the $h$-metric. 
First, observe that there exists some $l_1>0$ such that $U_s$ contains $U^h_{l_1s}$. By metric domination, we then see that
\begin{align}\label{eq:domination_areabound}
K_1 \Area_g(C\cap U_s)\geq \Area_h(C\cap U^h_{l_1s}).
\end{align}

By Proposition~\ref{prop:Duvalconstruction}, $\sqrt{h}$ is also plurisubharmonic. So the same argument as in \cite{duval2016result} tells us that the function $a(t)=\frac{1}{t}\Area_h(C\cap U^h_t)$ is monotone increasing with respect to $t$. In particular, we see that $a(t)\geq \lim_{t\to 0} a(t)$.

By construction, we can choose some constant $K_2$ so that $|\nabla h|\leq K_2\cdot t$ in $U^h_t$.
We claim that $\lim_{t\to 0}a(t)$ is bounded below by $K_2^{-1} \Length_h(\partial C\cap B^c)$. Indeed, for all $t\leq l_1 s$,
    \begin{align}\label{eq:areaslowerbound}
 \frac{1}{t}\int_{C\cap U^h_t} dd^c h &= \frac{1}{t^2}\int_{C\cap U^h_t} t dd^c h\nonumber \\&\geq \frac{1}{t^2}\int_{C\cap U^h_{t}\cap B^c} t dd^ch \geq \frac{{K_2}^{-1}}{t^2} \int_{C\cap U^h_t\cap B^c} \abs{\nabla h}dd^ch \nonumber\\
        &\geq \frac{K_2^{-1}}{t^2}\int_0^{t^2} \Length_h( C\cap \{h=\tau\}\cap B^c)d{\tau}.
    \end{align}

We can pass from the first line to the second line because the plurisubharmonicity of $h$ tells us that the integrand is non-negative. We get the second inequality on the second line from the final condition in Proposition~\ref{prop:Duvalconstruction}. Finally, to pass from the second line to the third line, we use the coarea formula. Since the limit of \eqref{eq:areaslowerbound} as $t\to 0$ is $K_2^{-1}\Length_h(\partial C\cap B^c)$, we get the desired identity. 

We now conclude the proof as follows. We can use the $C_2$-metric equivalence of $dd^ch(\cdot,\sqrt\cdot)$ and $g$ to pass to the $g$-length. Then using \cref{eq:domination_areabound}, we get \eqref{eq:our_reverseisoineq}. 
\end{proof}

\begin{remark}
As indicated in \cref{rem:immersed}, the proof can be adapted to immersed Lagrangians. Assume that the immersions are locally standard, and separated away from the intersection locus. We then only need to modify $\hpre$ on $N_s(L_i)$ away from $\bigcup_{ij} L_i\cap L_j$ to be $\tilde{\rho}_i$ away from the self-intersection locus and of the form $\absy\absx$ inside its standard charts. The set $B$ will then need to be a neighbourhood of the entire intersection locus, not just $\bigcup_{ij} L_i\cap L_j$. Then the proof of \cref{thm:reverseIsoperimetricInequality} carries over word-to-word the same. 
\end{remark}

\section{\texorpdfstring{Proof of \cref{prop:prod_psh-clean}}{Proof of PSH in local models}}
    \label{sec:pshProofs}
    We give here the full computations necessary to the proof of \cref{prop:prod_psh-clean} in the general clean case. To retrieve the transverse case, i.e.\ the case $k=0$, one can ignore the computations for $i\leq k$ and set $\delta_{i\leq k}=0$ and $\delta_{i>k}=1$ in the notation below. \par

To reduce the number of subscripts in this section, we adopt the notation that $\alpha:=\rho_1, \beta:=\rho_2$.
\label{subsec:cleanPrf}
We will need the following two technical lemmata, which we will prove later on. To enunciate them, we introduce the following notation:
\begin{align*}
        \delta_{i\leq k}=\begin{cases}
            1 &\text{ if }i\leq k \\
            0 &\text{ if }i> k \\
        \end{cases} \quad\text{and}\quad \delta_{i> k}=\begin{cases}
            1 &\text{ if }i> k \\
            0 &\text{ if }i\leq k \\
        \end{cases}.
    \end{align*}

\begin{lemma} \label{lem:ddcf}
    Set $v_0:=\sum_i (\alpha y_i\delta_{i>k}\frac{\del}{\del x_i}+(\beta-\alpha\delta_{i\leq k})x_i\frac{\del}{\del y_i})$ and $w_0:=-J_{\C^n}v_0$, where $J_{\C^n}$ is the ($2n\times 2n$)-matrix representing multiplication by $\sqrt{-1}$ in $\C^n=\RR^{2n}$. The matrix $M_0$ representing the form $dd^c\sqrt{\alpha\beta}(\cdot, \sqrt{-1}\cdot)$ in the standard basis $\{\frac{\del}{\del x_1},\dots,\frac{\del}{\del x_n},\frac{\del}{\del y_1},\dots,\frac{\del}{\del y_n}\}$ is equal to
    \begin{enumerate}[label=(\alph*)]
        \item $\displaystyle \frac{2\sum_{i\leq k}x_i^2}{\sqrt{\alpha\beta}}Id$ on $\mathrm{span}_\R\{v_0,w_0\}$;
        \item $\displaystyle \frac{\alpha+\beta}{\sqrt{\alpha\beta}}Id$ on the orthogonal complement $\mathrm{span}_\R\{v_0,w_0\}^\perp$
    \end{enumerate}
    outside of $L_1\cup L_2$. Here, we take the convention that $\sum_{i\leq k}x_i^2=0$ if $k=0$. In particular, $\sqrt{\alpha\beta}M_0$ is a multiple of the identity matrix precisely on $L_1\cup L_2\cup S_0$, where $S_0:=\{x_i=y_i=0|i>k\}$.
\end{lemma}

\begin{lemma} \label{lem:df-wedge-dcf}
    Set $v_1:=\sum_i (\alpha y_i\delta_{i>k}\frac{\del}{\del x_i}-(\beta+\alpha\delta_{i\leq k})x_i\frac{\del}{\del y_i})$ and $w_1:=J_{\C^n}v_1$. The matrix $M_1$ representing the form $(d\sqrt{\alpha\beta}\wedge d^c\sqrt{\alpha\beta})(\cdot, \sqrt{-1}\cdot)$ in the standard basis $\{\frac{\del}{\del x_1},\dots,\frac{\del}{\del x_n},$ $\frac{\del}{\del y_1},\dots,\frac{\del}{\del y_n}\}$ is equal to
    \begin{enumerate}[label=(\alph*)]
        \item $\displaystyle \left(2\sum_{i\leq k}x_i^2+\alpha+\beta\right)Id$ on $\mathrm{span}_\R\{v_1,w_1\}$;
        \item $0$ on the orthogonal complement $\mathrm{span}_\R\{v_1,w_1\}^\perp$.
    \end{enumerate}
    In particular, $M_1$ is $0$ precisely on $L_1\cup L_2$.
\end{lemma}

\begin{proof}[Proof of Proposition~\ref{prop:prod_psh-clean}]
    Since weak plurisubharmonicity of a function $f$ is equivalent to the positive semidefinitiveness of the form $dd^cf(\cdot, \sqrt{-1}\cdot)$, weak plurisubharmonicity of $\sqrt{\alpha\beta}$ follows directly from Lemma~\ref{lem:ddcf}. \par

    For the weak plurisubharmonicity of $\alpha\beta$, note that
    \begin{align}\label{eq:ddcf2}
    dd^c\alpha\beta=2\sqrt{\alpha\beta}\left(dd^c\sqrt{\alpha\beta}\right)+2\left(d\sqrt{\alpha\beta}\wedge d^c\sqrt{\alpha\beta}\right)
    \end{align}
    outside of $L_1\cup L_2$. One can also directly check that the formula holds also on $L_1\cup L_2$ by taking limits. Since the sum of positive semidefinite matrices is still positive semidefinite, weak plurisubharmonicity of $\alpha\beta$ then follows from Lemmata~\ref{lem:ddcf} and \ref{lem:df-wedge-dcf}. \par

    It also follows from (\ref{eq:ddcf2}) that there exists a vector $v\in T_{(x,y)}\RR^{2n}=\RR^{2n}$ such that $dd^c\alpha\beta(v,\sqrt{-1} v)=0$ if and only if the kernels of $\sqrt{\alpha\beta}M_0$ and $M_1$ intersect nontrivially. In view of Lemmata~\ref{lem:ddcf} and \ref{lem:df-wedge-dcf}, this means one of two things:
    \begin{enumerate}[label=(\alph*)]
        \item either $(x,y)\notin L_1\cup L_2\cup S_0$, $x_i=0$ for all $i\leq k$, and $\mathrm{span}_\R\{v_0,w_0\}\cap\mathrm{span}_\R\{v_1,w_1\}^\perp\neq \{0\}$ or
        \item $(x,y)\in L_1\cup L_2\cup S_0$ and $\alpha+\beta=0$.
    \end{enumerate}
    In option (b), note that $(x,y)\in L_1\cup L_2$ and $\alpha+\beta=0$ is equivalent to $(x,y)\in L_1\cap L_2$. But $(x,y)\in S_0$ and $\alpha+\beta=0$ is also equivalent to $(x,y)\in \{x_i=y_j=0|1\leq i\leq n,j>k\}=L_1\cap L_2$. Therefore, option (b) reduces to $(x,y)\in L_1\cap L_2$. \par

    Suppose now that we are in option (a). Since $\mathrm{span}_\R\{v_i,w_i\}$ is a 1-dimensional complex subspace, $\mathrm{span}_\R\{v_0,w_0\}\cap\mathrm{span}_\R\{v_1,w_1\}^\perp\neq \{0\}$ is equivalent to $v_0\perp v_1$, $v_0\perp w_1$, and $v_0,v_1\neq 0$. However, the fact that $(x,y)\notin L_1\cup L_2\cup S_0$ ensures precisely that the last condition is automatically satisfied. Therefore, option (a) reduces to the following equations: 
    \begin{align*}
    \begin{cases}
        x_i=0 \qquad\forall i\leq k \\
        \sum_{i>k} x_iy_i=0 \\
        \alpha=\beta
    \end{cases}
    \end{align*}
    since $\alpha,\beta\neq 0$ here. \par

    Noting that points respecting option (b) also respect these equations, we thus get that the set where $dd^c\alpha\beta$ is degenerate, i.e.\ where $\alpha\beta$ is not strictly plurisubharmonic, is the variety
    \begin{align*}
        V=\Big\{x_i=0,\ \sum_{j>k} x_jy_j=0,\ 
        \alpha=\beta\ \Big|\ i\leq k\Big\}.
    \end{align*}
\end{proof}

\begin{remark} \label{rem:shape_V}
    When $n-k\leq 1$, we simply have that $V=L_1\cap L_2$. However, when $n-k\geq 2$, $V$ will be a bigger $(k+2)$-codimensional variety of $\RR^{2n}$. For example, when $n-k=2$, it is the union of the two $n$-planes $\{x_i=0,x_n=\pm y_{n-1},y_n=\mp x_{n-1}|i\leq k\}$.
\end{remark}

\begin{proof}[Proof of Lemma~\ref{lem:ddcf}]
    The bilinear form $dd^c\sqrt{\alpha\beta}(\cdot, \sqrt{-1}\cdot)$ can be computed in coordinates to be
    \begin{align} \label{eq:sqrt_ab}
        &\left(\sqrt{\frac{\alpha}{\beta}}+\sqrt{\frac{\beta}{\alpha}}\right)\sum_{i=1}^n(dx_i\otimes dx_i+dy_i\otimes dy_i) \\
        &\quad +\sum_{i,j=1}^n\left[\left(\frac{\delta_{i\leq k}+\delta_{j\leq k}}{\sqrt{\alpha\beta}}-\sqrt{\frac{\alpha}{\beta^3}}\delta_{i\leq k}\delta_{j\leq k}\sqrt{\frac{\beta}{\alpha^3}}\right)x_ix_j-\sqrt{\frac{\alpha}{\beta^3}}y_iy_j\delta_{i>k}\delta_{j>k}\right] \nonumber\\
        &\qquad\qquad\quad \times (dx_i\otimes dx_j+dy_i\otimes dy_j) \nonumber\\
        &\quad +\sum_{i,j=1}^n\left[\left(\frac{1}{\sqrt{\alpha\beta}}-\sqrt{\frac{\alpha}{\beta^3}}\delta_{i\leq k}\right)x_iy_j\delta_{j>k}-\left(\frac{1}{\sqrt{\alpha\beta}}-\sqrt{\frac{\alpha}{\beta^3}}\delta_{j\leq k}\right)x_jy_i\delta_{i>k}\right] \nonumber\\
        &\qquad\qquad\quad \times (dx_i\otimes dy_j+dy_j\otimes dx_i). \nonumber
    \end{align}
    
    Putting the expression for $v_0$ in (\ref{eq:sqrt_ab}) gives, for $j\leq k$, that
    \begin{align*}
        dd^c\sqrt{\alpha\beta}\left(v_0, \sqrt{-1}\frac{\del}{\del x_j}\right)
        &= \sum_{i>k}\left(\sqrt{\frac{\alpha}{\beta}}-\sqrt{\frac{\beta}{\alpha}}\right)x_ix_jy_i+\sum_{i>k}\left(\sqrt{\frac{\beta}{\alpha}}-\sqrt{\frac{\alpha}{\beta}}\right)x_ix_jy_i=0 \\
        \intertext{and}
        dd^c\sqrt{\alpha\beta}\left(v_0, \sqrt{-1}\frac{\del}{\del y_j}\right)
        &= \left(\sqrt{\frac{\beta^3}{\alpha}}-\sqrt{\frac{\alpha^3}{\beta}}\right)x_j \\
        &\qquad +\sum_{i\leq k}\left(3\sqrt{\frac{\beta}{\alpha}}-3\sqrt{\frac{\alpha}{\beta}}+\sqrt{\frac{\alpha^3}{\beta^3}}-\sqrt{\frac{\beta^3}{\alpha^3}}\right)x_i^2x_j \\
        &\qquad +\sum_{i>k}\left[\left(\sqrt{\frac{\beta}{\alpha}}-\sqrt{\frac{\beta^3}{\alpha^3}}\right)x_i^2x_j-\left(\sqrt{\frac{\alpha}{\beta}}-\sqrt{\frac{\alpha^3}{\beta^3}}\right)x_jy_i^2\right] \\
        &= 2\left(\sqrt{\frac{\beta}{\alpha}}-\sqrt{\frac{\alpha}{\beta}}\right)x_j\sum_{i\leq k} x_i^2+\left(\sqrt{\frac{\beta^3}{\alpha}}-\sqrt{\frac{\alpha^3}{\beta}}\right)x_j \\
        &\qquad +\left(\sqrt{\frac{\beta}{\alpha}}-\sqrt{\frac{\beta^3}{\alpha^3}}\right)\alpha x_j-\left(\sqrt{\frac{\alpha}{\beta}}-\sqrt{\frac{\alpha^3}{\beta^3}}\right)\beta x_j \\
        &= \frac{2\sum_{i\leq k}x_i^2}{\sqrt{\alpha\beta}}b_j.
    \end{align*}

    Likewise, for $j>k$, one gets from (\ref{eq:sqrt_ab}) that
    \begin{align*}
        dd^c\sqrt{\alpha\beta}\left(v_0, \sqrt{-1}\frac{\del}{\del x_j}\right)
        &= \left(\sqrt{\frac{\alpha^3}{\beta}}+\sqrt{\alpha\beta}\right)y_j-\sum_{i\leq k}\left(\sqrt{\frac{\beta}{\alpha}}-2\sqrt{\frac{\alpha}{\beta}}+\sqrt{\frac{\alpha^3}{\beta^3}}\right)x_i^2y_j \\
        &\qquad -\sum_{i>k}\left[\sqrt{\frac{\beta}{\alpha}}x_ix_jy_i+\sqrt{\frac{\alpha^3}{\beta^3}}y_i^2y_j+\sqrt{\frac{\beta}{\alpha}}x_i^2y_j-\sqrt{\frac{\beta}{\alpha}}x_ix_jy_i\right] \\
        &= 2\sqrt{\frac{\alpha}{\beta}}y_j\sum_{i\leq k}x_i^2-\sqrt{\frac{\beta}{\alpha}}\alpha y_j-\sqrt{\frac{\alpha^3}{\beta^3}}\beta y_j+\sqrt{\frac{\alpha^3}{\beta}} y_j+\sqrt{\alpha\beta}y_j \\
        &= \frac{2\sum_{i\leq k} x_i^2}{\sqrt{\alpha\beta}}a_j
        \intertext{and}
        dd^c\sqrt{\alpha\beta}\left(v_0, \sqrt{-1}\frac{\del}{\del y_j}\right)
        &= \left(\sqrt{\alpha\beta}+\sqrt{\frac{\beta^3}{\alpha}}\right)x_j+\sum_{i\leq k}\left(\sqrt{\frac{\beta}{\alpha}}-\sqrt{\frac{\alpha}{\beta}}-\sqrt{\frac{\beta^3}{\alpha^3}}+\sqrt{\frac{\beta}{\alpha}}\right)x_i^2x_j \\
        &\qquad -\sum_{i>k}\left[\sqrt{\frac{\beta^3}{\alpha^3}}x_i^2x_j+\sqrt{\frac{\alpha}{\beta}}x_iy_iy_j-\sqrt{\frac{\alpha}{\beta}}x_iy_iy_j+\sqrt{\frac{\alpha}{\beta}}x_jy_i^2\right] \\
        &= 2\sqrt{\frac{\beta}{\alpha}}x_j\sum_{i\leq k} x_i^2+\left(\sqrt{\alpha\beta}+\sqrt{\frac{\beta^3}{\alpha}}\right)x_j -\sqrt{\frac{\alpha}{\beta}}\beta x_j-\sqrt{\frac{\beta^3}{\alpha^3}}\alpha x_j \\
        &= \frac{2\sum_{i\leq k}x_i^2}{\sqrt{\alpha\beta}}b_j.
    \end{align*}
    
    Therefore, when nonzero, $v_0$ is an eigenvector of $M_0$ with associated eigenvalue $\frac{2\sum_{i\leq k}x_i^2}{\sqrt{\alpha\beta}}$. But note that $M_0$ commutes with $J_0$, the $2n\times 2n$ matrix representing multiplication by $i$, since $dd^c\sqrt{\alpha\beta}\left(\cdot, \sqrt{-1}\cdot\right)$ is a $\C$-sesquilinear form. Therefore, $w_0:=-J_0v_0=\sum_i((\beta-\alpha\delta_{i\leq k})x_i\frac{\del}{\del x_i}-\alpha y_i\delta_{i>k}\frac{\del}{\del y_i})$ must also be an eigenvector with the same eigenvalue. \par

    Suppose now that $v=\sum_j(a_j\frac{\del}{\del x_j}+b_j\frac{\del}{\del y_j})$ is orthogonal to both $v_0$ and $w_0$, i.e.\ 
    \begin{align}
        \sum_{i>k}(\alpha y_i\alpha_i+\beta x_ib_i)+(\beta-\alpha)\sum_{i\leq k} x_ib_i&=0 \label{eq:v0-ortho}
        \intertext{and}
        (\beta-\alpha)\sum_{i\leq k}x_ia_i+\sum_{i>k}(\beta x_ia_i-\alpha y_ib_i)&=0. \label{eq:w0-ortho}
    \end{align}
    Denote by $m$ the matrix representing the form $dd^c\sqrt{\alpha\beta}\left(\cdot, \sqrt{-1}\cdot\right)-(\sqrt{\frac{\alpha}{\beta}}+\sqrt{\frac{\beta}{\alpha}})(\cdot,\cdot)$, where $(\cdot,\cdot)$
    is the usual inner product. Using again (\ref{eq:sqrt_ab}), we get for $j\leq k$ that
    \begin{align*}
        \left(mv,\frac{\del}{\del x_j}\right)
        &= \sum_{i\leq k}\left(\frac{2}{\sqrt{\alpha\beta}}-\sqrt{\frac{\alpha}{\beta^3}}-\sqrt{\frac{\beta}{\alpha^3}}\right)x_ix_ja_i
        +\sum_{i>k}\left(\frac{1}{\sqrt{\alpha\beta}}-\sqrt{\frac{\beta}{\alpha^3}}\right)x_ix_ja_i \\
        &\qquad +\sum_{i>k}\left(\frac{1}{\sqrt{\alpha\beta}}-\sqrt{\frac{\alpha}{\beta^3}}\right)x_jy_ib_i \\
        &= \sum_{i\leq k}\left(\frac{2}{\sqrt{\alpha\beta}}-\sqrt{\frac{\alpha}{\beta^3}}-\sqrt{\frac{\beta}{\alpha^3}}\right)x_ix_ja_i-\sum_{i\leq k}(\beta-\alpha)\left(\frac{1}{\sqrt{\alpha\beta^3}}-\frac{1}{\sqrt{\alpha^3\beta}}\right)x_ix_ja_i \\
        &=0,
    \end{align*}
    where we have used (\ref{eq:w0-ortho}) to get the second equality. We analogously get $(mv,\frac{\del}{\del y_j})=0$ from (\ref{eq:v0-ortho}). For $j>k$, we instead have
    \begin{align*}
        \left(mv,\frac{\del}{\del x_j}\right)
        &= \sum_{i\leq k}\left(\frac{1}{\sqrt{\alpha\beta}}-\sqrt{\frac{\beta}{\alpha^3}}\right)x_ix_ja_i-\sum_{i>k}\left(\sqrt{\frac{\beta}{\alpha^3}}x_ix_j+\sqrt{\frac{\alpha}{\beta^3}}y_iy_j\right)a_i \\
        &\qquad -\sum_{i\leq k}\left(\frac{1}{\sqrt{\alpha\beta}}-\sqrt{\frac{\beta}{\alpha^3}}\right)x_iy_jb_i-\sum_{i>k}\left(\frac{x_iy_j}{\sqrt{\alpha\beta}}-\frac{x_jy_i}{\sqrt{\alpha\beta}}\right)b_i \\
        &= \sum_{i\leq k}\left[\left(\frac{1}{\sqrt{\alpha\beta}}-\sqrt{\frac{\beta}{\alpha^3}}\right)(x_ix_ja_i-x_iy_jb_i)
        +\frac{\beta-\alpha}{\sqrt{\alpha^3\beta}}x_ix_ja_i
        +\frac{\beta-\alpha}{\sqrt{\alpha\beta^3}}x_iy_jb_i\right] \\
        &=0,
    \end{align*}
    where we get the second equality using both (\ref{eq:v0-ortho}) and (\ref{eq:w0-ortho}). We get that $(mv,\frac{\del}{\del y_j})=0$ similarly.

    In other words, when restricted to the orthogonal complement of $\mathrm{span}_\R\{v_0,w_0\}$, the form $dd^c\sqrt{\alpha\beta}\left(\cdot, \sqrt{-1}\cdot\right)$ is just $(\sqrt{\frac{\alpha}{\beta}}+\sqrt{\frac{\beta}{\alpha}})(\cdot,\cdot)$. This proves the first part of the lemma. \par

    For the second part, note that there are two ways in which $\sqrt{\alpha\beta} M_0$ becomes a multiple of the identity: either both possible eigenvalues become the same, or $v_0=0$. The first situation happens precisely on $S_0$, while the second one happens precisely on $L_1\cup L_2\cup(S_0\cap\{\alpha=\beta\})$. The union of these spaces is, of course, $L_1\cup L_2\cup S_0$.
\end{proof}

\begin{proof}[Proof of Lemma~\ref{lem:df-wedge-dcf}]
    The proof follows the same structure as that of Lemma~\ref{lem:ddcf}; we give here the details. The bilinear form $d\sqrt{\alpha\beta}\wedge d^c\sqrt{\alpha\beta}(\cdot, \sqrt{-1}\cdot)$ can be computed in coordinates to be
    \begin{small}
    \begin{align}  \label{eq:df-wedge-dcf}
        &\sum_{i,j=1}^n\left[\left(\frac{\beta}{\alpha}+\delta_{i\leq k}+\delta_{j\leq k}+\frac{\alpha}{\beta}\delta_{i\leq k}\delta_{j\leq k}\right)x_ix_j+\frac{\alpha}{\beta}y_iy_j\delta_{i>k}\delta_{j>k}\right](dx_i\otimes dx_j+dy_i\otimes dy_j) \\
        &\quad +\sum_{i,j=1}^n\left[\left(1+\frac{\alpha}{\beta}\delta_{i\leq k}\right)x_iy_j\delta_{j>k}-\left(1+\frac{\alpha}{\beta}\delta_{j\leq k}\right)x_jy_i\delta_{i>k}\right](dx_i\otimes dy_j+dy_j\otimes dx_i). \nonumber
    \end{align}
    \end{small}
    
    Putting the expression for $v_1$ in (\ref{eq:df-wedge-dcf}) gives, for $j\leq k$, that
    \begin{align*}
        d\sqrt{\alpha\beta}\wedge d^c\sqrt{\alpha\beta}\left(v_1, \sqrt{-1}\frac{\del}{\del x_j}\right)
        &= \sum_{i>k}\left(\frac{\beta}{\alpha}+1\right)\alpha x_ix_jy_i-\sum_{i>k}\left(1+\frac{\alpha}{\beta}\right)\beta x_ix_jy_i=0 \\
        \intertext{and}
        d\sqrt{\alpha\beta}\wedge d^c\sqrt{\alpha\beta}\left(v_0, \sqrt{-1}\frac{\del}{\del y_j}\right)
        &= -\sum_{i\leq k}\left(\frac{\beta}{\alpha}+2+\frac{\alpha}{\beta}\right)(\alpha+\beta)x_i^2x_j
        -\sum_{i>k}\left(\frac{\beta}{\alpha}+1\right)\beta x_i^2x_j \\
        &\qquad -\sum_{i>k}\left(1+\frac{\alpha}{\beta}\right)\alpha x_jy_i^2 \\
        &=-2(\alpha+\beta)\sum_{i\leq k}x_i^2x_j-\left(\beta+\frac{\beta^2}{\alpha}\right)\alpha x_j-\left(\frac{\alpha^2}{\beta}+\alpha\right)\beta x_j \\
        &=\left(2\sum_{i\leq k}x_i^2+\alpha+\beta\right)b_j.
    \end{align*}

    Likewise, for $j>k$, one gets from (\ref{eq:df-wedge-dcf}) that
    \begin{align*}
        d\sqrt{\alpha\beta}\wedge d^c\sqrt{\alpha\beta}\left(v_1, \sqrt{-1}\frac{\del}{\del x_j}\right)
        &= \sum_{i>k}\left[\frac{\beta}{\alpha}\alpha x_ix_jy_i+\frac{\alpha}{\beta}\alpha y_i^2y_j+\beta x_i^2y_j-\beta x_iy_jy_i\right] \\
        &\qquad +\sum_{i\leq k}\left(1+\frac{\alpha}{\beta}\right)(\alpha+\beta)x_i^2y_j \\
        &= 2\alpha\sum_{i\leq k}x_i^2y_j+\alpha\beta y_j+\alpha^2y_j \\
        &= \left(2\sum_{i\leq k}x_i^2+\alpha+\beta\right)a_j
        \intertext{and}
        d\sqrt{\alpha\beta}\wedge d^c\sqrt{\alpha\beta}\left(v_1, \sqrt{-1}\frac{\del}{\del y_j}\right)
        &= -\sum_{i\leq k}\left(\frac{\beta}{\alpha}+1\right)(\alpha+\beta)x_i^2x_j \\
        &\qquad -\sum_{i>k}\left[\frac{\beta}{\alpha}\beta x_i^2x_j+\frac{\alpha}{\beta}\beta x_iy_iy_j-\alpha x_iy_iy_j+\alpha x_jy_i^2\right] \\
        &= -2\beta\sum_{i\leq k}x_i^2x_j-\frac{\beta^2}{\alpha}\alpha x_j-\alpha\beta x_j \\
        &= \left(2\sum_{i\leq k}x_i^2+\alpha+\beta\right)b_j.
    \end{align*}
    Therefore, when nonzero, $v_1$ is an eigenvector of $M_1$ with associated eigenvalue $(2\sum_{i\leq k}x_i^2+\alpha+\beta)$. But note that $M_1$ commutes with $J_0$ since $d\sqrt{\alpha\beta}\wedge d^c\sqrt{\alpha\beta}\left(\cdot, \sqrt{-1}\cdot\right)$ is a $\C$-sesquilinear form. Therefore, $w_1:=J_0v_1=\sum_i((\beta+\alpha\delta_{i\leq k})x_i\frac{\del}{\del x_i}+\alpha y_i\delta_{i>k}\frac{\del}{\del y_i})$ must also be an eigenvector with the same eigenvalue. \par

    Suppose now that $v=\sum_j(a_j\frac{\del}{\del x_j}+b_j\frac{\del}{\del y_j})$ is orthogonal to both $v_1$ and $w_1$, i.e.\ 
    \begin{align}
        \sum_{i>k}(\alpha y_i\alpha_i-\beta x_ib_i)-(\alpha+\beta)\sum_{i\leq k} x_ib_i&=0 \label{eq:v1-ortho}
        \intertext{and}
        (\alpha+\beta)\sum_{i\leq k}x_ia_i+\sum_{i>k}(\beta x_ia_i+\alpha y_ib_i)&=0. \label{eq:w1-ortho}
    \end{align}
    Using again (\ref{eq:df-wedge-dcf}), we get for $j\leq k$ that
    \begin{align*}
        d\sqrt{\alpha\beta}\wedge d^c\sqrt{\alpha\beta}\left(v,\sqrt{-1}\frac{\del}{\del x_j}\right)
        &= \sum_{i\leq k}\left(\frac{\beta}{\alpha}+2+\frac{\alpha}{\beta}\right)x_ix_ja_i
        +\sum_{i>k}\left(\frac{\beta}{\alpha}+1\right)x_ix_ja_i \\
        &\qquad +\sum_{i>k}\left(1+\frac{\alpha}{\beta}\right)x_jy_ib_i \\
        &= \sum_{i\leq k}\left(\frac{\beta}{\alpha}+2+\frac{\alpha}{\beta}\right)x_ix_ja_i-\sum_{i\leq k}(\alpha+\beta)\left(\frac{1}{\alpha}+\frac{1}{\beta}\right)x_ix_ja_i \\
        &=0,
    \end{align*}
    where we have used (\ref{eq:w1-ortho}) to get the second equality. We analogously get $d\sqrt{\alpha\beta}\wedge d^c\sqrt{\alpha\beta}(v,\frac{\del}{\del y_j})=0$ from (\ref{eq:v1-ortho}). For $j>k$, we instead have
    \begin{align*}
        d\sqrt{\alpha\beta}\wedge d^c\sqrt{\alpha\beta}\left(v,\sqrt{-1}\frac{\del}{\del x_j}\right)
        &= \sum_{i\leq k}\left(\frac{\beta}{\alpha}+1\right)x_ix_ja_i+\sum_{i>k}\left[\frac{\beta}{\alpha}x_ix_j+\frac{\alpha}{\beta}y_iy_j\right]a_i \\
        &\qquad -\sum_{i\leq k}\left(1+\frac{\alpha}{\beta}\right)x_iy_jb_i-\sum_{i>k}\left[x_iy_j-\frac{\alpha}{\beta}x_jy_i\right]b_i \\
        &=0,
    \end{align*}
    where we get the second equality using both (\ref{eq:v1-ortho}) and (\ref{eq:w1-ortho}). We get that $d\sqrt{\alpha\beta}\wedge d^c\sqrt{\alpha\beta}(v,\sqrt{-1}\frac{\del}{\del y_j})=0$ similarly.

    In other words, when restricted to the orthogonal complement of $\mathrm{span}_\R\{v_1,w_1\}$, the form $d\sqrt{\alpha\beta}\wedge d^c\sqrt{\alpha\beta}(\cdot, \sqrt{-1}\cdot)$ is just $0$. This proves the first part of the lemma. \par

    For the second part, note that there are two ways in which $M_1$ becomes the $0$ matrix: either both possible eigenvalues become the same, or $v_1=0$. The first situation happens precisely on $L_1\cap L_2$, while the second one happens precisely on $L_1\cup L_2$. The union of these spaces is, of course, $L_1\cup L_2$.
\end{proof} 

\newpage
\printbibliography 

\Addresses

\end{document}